\documentclass[12pt]{article}
\usepackage{bbm}
\usepackage{amsmath}
\usepackage[hidelinks]{hyperref}
\allowdisplaybreaks

\renewcommand{\emph}[1]{\textit{#1}}
\usepackage{enumerate,amsmath,amsthm,latexsym,amssymb}
\usepackage{color}\usepackage{graphicx}

\def\nn{\nonumber}

\definecolor{brown}{cmyk}{0, 0.72, 1, 0.45}
\definecolor{grey}{gray}{0.5}

\newcommand{\old}[1]{}

\newcounter{rot}

\newcommand{\card}[1]{\left|#1\right|}

\newcommand{\ignore}[1]{}

\newcommand{\set}[1]{\left\{#1\right\}}

\def\cP{\mathcal{P}}
\def\cQ{\mathcal{Q}}

\def\ii_(#1,#2){i_{#1}^{#2}}

\def\bx{{\bf x}}

\def\d{\delta}

\def\e{\varepsilon}
\def\f{\phi}
\def\F{\Phi}
\def\g{\gamma}
\def\G{\Gamma}

\def\l{\lambda}

\def\n{\nu}
\def\p{\pi}
\def\P{\Pi}
\def\r{\rho}

\def\s{\sigma}

\def\up{\upsilon}

\def\1{{\bf 1}}
\def\0{{\bf 0}}

\def\cE{\mathcal{E}}

\def\cT{\mathcal{T}}

\newcommand{\brac}[1]{\left( #1 \right)}

\def\E{{\bf E}}

\renewcommand{\Pr}{\operatorname{\bf Pr}}
\newcommand\bfrac[2]{\left(\frac{#1}{#2}\right)}

\def\bx{{\bf x}}

\def\bd{{\bf d}}

\def\2G{{\sc 2greedy}}

\newcommand{\nospace}[1]{}

\def\path{\operatorname{PATH}}

\def\bd{{\bf d}}

\newtheorem{theorem}{Theorem}[section]

\newtheorem{lemma}[theorem]{Lemma}
\newtheorem{corollary}[theorem]{Corollary}

\newtheorem{remthm}[theorem]{Remark}

\newcounter{thmtemp}

\usepackage[ruled, linesnumbered]{algorithm2e}
\def\gnm3{G_{n,m}^{\delta\geq 3}}
\def\Gnm3{{\mathcal G}_{n,m}^{\delta\geq 3}}
\newcommand{\beq}[2]{\begin{equation}\label{#1}#2\end{equation}}
\def\G{\Gamma}
\def\cG{{\mathcal G}}

\def\cH{{\mathcal H}}

\def\ec{3}
\def\ecec{6}

\newcommand{\multstar}[1]{\begin{multline*}#1\end{multline*}}
\def\bsx{\boldsymbol \xi}
\begin{document}
\title{A scaling limit for the length of the longest cycle in a sparse random graph}
\author{ Michael Anastos\thanks{Department of Mathematics and Computer Science, Freie Universit{\"a}t Berlin, Berlin, Germany, email:{\bf manastos@zedat.fu-berlin.de}} and Alan Frieze\thanks{Department of Mathematical Sciences, Carnegie Mellon University, Pittsburgh PA, U.S.A. email:{\bf alan@random.math.cmu.edu}; the author is supported in part by NSF Grant DMS1363136} }

\maketitle
\begin{abstract}
We discuss the length of the longest cycle in a sparse random graph $G_{n,p},p=c/n$. $c$ constant. We show that for large $c$ there exists a function $f(c)$ such that $L_{c,n}/n\to f(c)$ a.s. The function $f(c)=1-\sum_{k=1}^\infty p_k(c)e^{-kc}$ where $p_k$ is a polynomial in $c$. We are only able to explicitly give the values $p_1,p_2$, although we could in principle compute any $p_k$. We see immediately that the length of the longest path is also asymptotic to $f(c)n$ w.h.p. 
\end{abstract}
\section{Introduction}
There are several basic questions that can be asked in the context of a class of graphs. E.g. what is the chromatic number? Is the graph Hamiltonian? Another such basic question is the following: how long is the longest cycle? In this paper we study this question in relation to the sparse random graph $G_{n,p},p=c/n$ for a constant $c>0$. Thus, let $L_{c,n}$ denote the length of the longest cycle in the random graph $G_{n,c/n}$. Erd\H{o}s \cite{Erdos1} conjectured that if $c>1$ then w.h.p. $L_{c,n}\geq \ell(c)n$ where $\ell(c)>0$ is independent of $n$. This was proved by Ajtai, Koml\'os and Szemer\'edi \cite{AKS} and in a slightly weaker form by de la Vega \cite{Vega} who proved that if $c>4\log 2$ then $f(c)=1-O(c^{-1})$. See also Suen \cite{Suen}. Although this answered Erd\H{o}s's question it only gives us a lower bound for the length of the longest cycle. Bollob\'as \cite{Bopath1} realised that for large $c$ one could find a large path/cycle w.h.p. by concentrating on a large subgraph with large minimum degree and demonstrating Hamiltonicity. In this way he showed that $\ell(c)\geq 1-c^{24}e^{-c/2}$. This was then improved by Bollob\'as, Fenner and Frieze \cite{BFF2} to $\ell(c)\geq 1-c^6e^{-c}$ and then by Frieze \cite{Fpath} to $\ell(c)\geq 1-(1+\e_c)(1+c)e^{-c}$ where $\e_c\to0$ as $c\to\infty$. This last result is optimal up to the value of $\e_c$, as there are w.h.p. $\approx (1+c)e^{-c}n$ vertices of degree 0 or 1. 

The basic open question to this point, is at to whether or not there exists a function $f(c)$ such that w.h.p. the $L_{c,n}=(1+\e_n)f(c)n$ where $\e_n\to0$ as $n\to 0$. And what is $f(c)$. In this paper we establish the existence of $f(c)$ for large $c$ and give a method of computing it to arbitrary accuracy. We note that this is one case of a fundamental extremal random variable where the existence of a scaling limit has not previously been shown to exist and does not appear to be susceptible to the interpolation method as in Bayati, Gamarnik and Tetali \cite{BGT}. 

Let $p=c/n$ and let $G=G_{n,p}$. We will assume throughout that $c$ is sufficiently large. Let $C_2$ denote the 2-core of $G$. By this we mean that part of the giant component consisting of vertices that are in at least one cycle. The longest cycle in $G$ is contained in $C_2$ and the length of the longest path in $G_{n,c/n}$ differs from this by $O(\log n)$ w.h.p. This will be for two reasons. The first reason is that we will establish a Hamiltonian subgraph of $C_2$ that contains the longest path in $C_2$ and the second reason for this is that w.h.p. the giant component of $G$ consists of $C_2$ plus a forest of trees with maximum diameter  $O(\log n)$.

As in the papers, \cite{Bopath1}, \cite{BFF2} and \cite{Fpath} we consider a process that builds a large Hamiltonian subgraph. We construct a sequence of sets $S_0=\emptyset,S_1,S_2,\ldots,S_L\subseteq C_2$ and their  induced subgraphs $\G_0, \G_1,\G_2,\ldots,\G_L$. Suppose now that we have constructed $S_\ell$, $\ell \geq 0$. We construct $S_{\ell+1}$ from $S_{\ell}$ via one of two cases: 

{\bf Construction of $\G_L$}\\
{\bf Case a:} If there is $v\in S_\ell$ that has exactly one or two neighbors $W$ in $ C_2 \setminus S_\ell$, then we add $W$ to $S_\ell$ to make $S_{\ell+1}$. \\
{\bf Case b:} If there is a vertex $v\in C_2\setminus S_\ell$ 
that has at most two neighbors in $C_2\setminus S_\ell$
then we define $S_{\ell+1}$ to be $S_\ell$ plus $v$ plus  the neighbors of $v$ in $C_2 \setminus S_\ell$.

$S_L$ is the set we end up with when there are no more vertices to add. We note that $S_L$ is well-defined and does not depend on the order of adding vertices. Indeed, suppose we have two distinct outcomes $O_1= v_1,v_2,\ldots,v_r$ and $O_2= w_1,w_2.,\ldots,w_s$. Assume without loss of generality that there exists $i$ which is the smallest index such that $w_i\notin O_1$. Then, $X=\set{w_1,w_2,\ldots,w_{i-1}}\subseteq Y=\set{v_1,v_2,\ldots,v_r}$. If $w_i$ was added in Step a as the neighbor of $v\in S_\ell=X$ then $v\in Y$  and $v$ has at most two neighbors in $C_2\setminus Y$. This contradicts the fact that $w_i\notin Y$. Suppose then that $w_i$ is added in Step b. If $w_i=v$ then it has at most two neighbors in $C_2\setminus X$ and hence it has at most two neighbors in $C_2\setminus Y$. This contradicts the fact that $w_i\notin Y$. If $w_i$ is the neighbor of $v\in X\subseteq Y$ then we get the same contradiction. It follows that $\set{w_1,w_2.,\ldots,w_s}\subseteq \set{v_1,v_2,\ldots,v_r}$ and vice-versa, by the same reasoning.

We will argue below in Section \ref{sizes} that w.h.p. the graph $\G_L$ induced by $S_L$ is a forest plus a few small components. Each tree in $\G_L$ will w.h.p. have at most $\log n$ vertices. For a tree component $T$ let 
\[
\upsilon_0(T)\text{ denote the set of vertices of $T$ that have no neighbors outside $S_L$}.
\]
{\bf Notation 1:} Let $\cT$ denote the set of trees in $\G_L$. For a tree $T\in \cT$ let $\cP_T$ be the set of vertex disjoint path packings of $T$ where we allow only paths whose start- and end- vertex are have neighbors in $C_2 \setminus V(T)$.
Here we allow paths of length 0, so that a single vertex with neighbors in $C_2 \setminus V(T)$ counts as a path. For $P\in \cP_T$ let $n(T,P)$ be the number of vertices in $T$ that are not covered by $P$. Let $\f(T)=\min_{P\in \cP_T} n(T,P)$ and $\cQ(T)\in \cP$ denote a set of paths that leaves $\f(T)$ vertices of $T$ uncovered i.e. satisfies $n(T,Q(T))=\f(T)$.

If $A=A(n), B=B(n)$ then we write $A\approx B$ if $A=(1+o(1))B$ as $n\to\infty$.

We will prove
\begin{theorem}\label{th1}
Let $p=c/n$ where $c>1$ is a sufficiently large constant. Then w.h.p.
\beq{sizeC}{
L_{c,n}\approx |V(C_2)|-\sum_{T\in \cT}\f(T).
}
\end{theorem}
The size of $C_2$ is well-known. Let $x$ be the unique solution of $xe^{-x}=ce^{-c}$ in $(0,1)$. Then w.h.p. (see e.g. \cite{FK}, Lemma 2.16),
\begin{align}
|C_2|&\approx (1-x)\brac{1-\frac{x}{c}}n.\label{C2v}\\
|E(C_2)|&\approx\brac{1-\frac{x}{c}}^2\frac{c}{2}n.\label{C2e}
\end{align}
Equation (4.5) of Erd\H{o}s and R\'enyi \cite{ER} tells us that 
\beq{xval}{
x=\sum_{k=1}^\infty\frac{k^{k-1}}{k!}(ce^{-c})^k=ce^{-c}+c^2e^{-2c}+O(c^3e^{-3c}).
}
We will argue below that w.h.p., as $c$ grows, that
\beq{small}{
\sum_{T\in \cT}\f(T)=O(c^{6}e^{-3c})n.
}
We therefore have the following improvement to the estimate in \cite{Fpath}.
\begin{corollary}\label{cor1} 
W.h.p., as $c$ grows, that
\beq{sizeC1}{
L_{c,n}\approx \brac{1-(c+1)e^{-c}-c^2e^{-2c}+ O(c^{6}e^{-3c})}n.
}
\end{corollary}
Note the term $(c+1)e^{-c}$ which accounts for vertices of degree 0 or 1. In principle we can compute more terms than what is given in \eqref{sizeC1}. We claim next that there exists some function $f(c)$ such that the sum in \eqref{sizeC} is concentrated around $f(c)n$. In other words, the sum in \eqref{sizeC} has the form  $\approx f(c)n$ w.h.p.
\begin{theorem}\label{limit}
\begin{enumerate}[(a)]
\item There exists a function $f(c)$ such that for any $\epsilon>0$, there exists $n_\e$ such that for $n\geq n_\e$,
\beq{eq:expectation}{
\card{\frac{\E[L_{c,n}]}{n}-f(c)}\leq \epsilon.
}
\item
\[
\frac{L_{c,n}}{n}\to f(c)\ a.s.
\]
\end{enumerate}
\end{theorem}
We will prove Theorem \ref{limit} in Section \ref{seclimit}.
\subsection{Structure of $\G_L$:}\label{sizes}
We first bound the size of $S_L$. We need the following lemma on the density of small sets.
\begin{lemma}\label{lem1}
W.h.p., every set $S\subseteq [n]$ of size at most $n_0=n/10c^3$ contains less than $3|S|/2$ edges in $G_{n,p}$.
\end{lemma}
\begin{proof}
The expected number of sets invalidating the claim can be bounded by
\[
\sum_{s=4}^{n_0}\binom{n}{s}\binom{\binom{s}{2}}{3s/2}\bfrac{c}{n}^{3s/2}\leq \sum_{s=4}^{n_0}\brac{\frac{ne}{s}\cdot\bfrac{se}{3}^{3/2}\cdot\bfrac{c}{n}^{3/2}}^s=\sum_{s=4}^{n_0} \bfrac{e^{5/2}c^{3/2}s^{1/2}}{3^{3/2}n^{1/2}}^s=o(1).
\]
\end{proof}
Now consider the construction of $S_L$. Let $A$ be the set of the vertices with degree less than $D=100$ and let $S_0'=(A \cup N(A) )\cap S_L\subseteq S_L$. If we start with $S_0=S_0'$ and run the process for constructing $\Gamma_L$ then we will producee the same $S_L$ as if we had started with $S_0= \emptyset$. This is because, as we have shown, the order of adding vertices does not matter. Now w.h.p. there are at most $n_D=\frac{2c^De^{-c}}{D!}n$ vertices of degree at most $D$ in $G_{n,p}$, (see for example Theorem 3.3 of \cite{FK}) and so $|S_0'|\leq  Dn_D$. 

Now suppose that the process runs for another $k$ rounds. Then $S_{k}$ has a least $kD/2$ edges and at most $Dn_D+3k$ vertices. This is because round $k$ adds at most three {\em new} vertices to $S_k$ and the $k$ vertices that take the role of $v$ have degree at least $D$ and all of their neighbors will be in $S_k$. If $k$ reaches $4n_D$ then 
\[
\frac{e(S_k)}{|S_k|}\geq  \frac{4Dn_d}{2}\cdot \frac{1}{(D+12)n_d}>\frac32.
\]
So, by Lemma \ref{lem1}, we can assert that w.h.p. the process runs for less than $4n_D$ rounds and,
\beq{smallS}{
|V(\G_L)|\leq (D+4)n_D\leq ne^{-c/2}.
}
We note the following properties of $S_L$. Let 
\[
V_1=C_2\setminus S_L\text{ and }V_2=\set{v\in  S_L:\;v\text{ has at least one neighbor in }V_1} .
\]
Then, 
\begin{enumerate}[{\bf G1}]
\item Each vertex $v\in S_L \setminus V_2$ has no neighbors in $V_1$.
\item Each $v\in V_1 \cup V_2$ has at least $\ec$ neighbors in $V_1$.
\end{enumerate}
Given the definition of $V_2$, for $T\in \cT$ we can express $\up_0(T)$ as
\[
\up_0(T)=V(T)\setminus V_2.
\]

We will now show that w.h.p. each component $K$ of $\Gamma_L$  satisfies 
\beq{vK}{
|\up_0(K)|\geq 
 \frac{|V(K)|}{\ec}.
}
We will prove that for $0\leq i \leq L$ and each component $K$ spanned by $S_i$, 
\beq{vK1}{
|\up_{0,i}(K)|\geq 
 \frac{|V(K)|}{\ec}.
}
Here $v_{0,i}(K)$ is taken to be the number of vertices in $V(K)$ with no neigbors in $C_2\setminus K$. Taking $i=L$ in \eqref{vK1} yields \eqref{vK}. We proceed by an induction on $i$.

$S_0=\emptyset$ and so for $i=0$, \eqref{vK1} is satisfied by every component spanned by $S_0$. Suppose that at step $i=\ell$,  \eqref{vK1} is satisfied by every component spanned by $S_\ell$.

At step $\ell+1$, assume that $v$ invokes either Case a or Case b. In both cases  $S_{\ell+1}=S_\ell \cup \big(\{v\} \cup N(v)\big).$ The addition of the new vertices into $S_\ell$  could merge components $K_1,K_2,\ldots,K_r$  into one component $K'$ while adding at most $\ec$ vertices. Hence  $\ec+\sum_{j\in [r]} |K_i| \geq |K'|$. In addition every vertex that contributed to $v_{0,\ell}(K_j)$, $j=1,2,...,r$ now contributes towards $v_{0,\ell+1}(K')$.
Also $v$ has neighbors outside $S_{\ell}$ but no neigbors outside $S_{\ell+1}$. The inductive hypothesis implies that $\up_{0,\ell}(K_j) \geq |K_j|/\ec$ for $j\in [r]$. Thus,
\[
\up_{0,\ell+1}(K')\geq 1+\sum_{j \in [r]}\up_{0,\ell}(K_j) \geq 1+\frac{1}{ \ec}\sum_{j \in [r]}|K_j|\geq 1+\frac{|K'|-\ec}{\ec} = \frac{|K'|}{\ec}.
\]
and so \eqref{vK1} continues to hold for all the components spanned by $S_{\ell+1}$.

We next show that w.h.p., only a small component $K$ can satisfy \eqref{vK}. We consider $K$ in the context of $G_{n,p}$ in which case $K$ will have at least $|V(K)|/3$ vertices with no neighbors outside $K$. So, the expected number of components of size $k\leq ne^{-c/2}$ that satisfy this condition is at most
\begin{align}\label{11}
\binom{n}{k}k^{k-2}\bfrac{c}{n}^{k-1}\binom{k}{k/ \ec} (1-p)^{k(n-k)/\ec}&
\leq \bfrac{ne}{k}^kk^{k-2}\bfrac{c}{n}^{k-1}2^ke^{- ck/\ecec} \nonumber
\\&\leq \frac{n}{ck^2}\brac{2ce^{1-c/\ecec}}^k=o(n^{-2}),
\end{align}
if $c$ is large and $k\geq \log n$.

So, we can assume that all components are of size at most $\log n$. Then the expected number of vertices on components that are not trees is bounded by
\begin{align*}
\sum_{k=3}^{\log n}\binom{n}{k}k^{k+1}\bfrac{c}{n}^{k}\binom{k}{k/\ec} (1-p)^{k(n-k)/\ec}&
\leq \sum_{k=3}^{\log n}\bfrac{ne}{k}^kk^{k+1}\bfrac{c}{n}^{k}(e^{-ck/\ecec})
\\&\leq \sum_{k=3}^{\log n}k\brac{2ce^{1-c/\ecec}}^k=O(1).
\end{align*}
Markov's inequality implies that w.h.p. such components span at most $\log n=o(n)$ vertices.
\vspace{3mm}
\\{\bf Notation 2:}
For $T\in\cT$,  let $M_T$ be the matching on $V_2$ obtained by replacing each path of $\cQ(T)$ of length at least 1 by an edge and let $M^*=\bigcup_{T\in \cT}M_T$.  Let $I(T)$ denote the internal vertices of the paths $\cQ(T)$ and $I^*=\bigcup_{T\in \cT}I_T^*$ and $V_2^*=V_2\setminus I^*$. We let $\G^*_1$ be the subgraph of $G$ induced by $V_1$. We also let $ \G^*_2 $ be the bipartite graph with vertex partition  $V_1,V_2^*$ and all edges $\{e \in E(G):e \in V_1 \times V_2^*\}$. Finally let $\G^*=\G_1^*\cup \G^*_2\cup M^*$ and $V^*=V_1\cup V_2^*=V(\G^*)$. 
\section{Proof of Theorem \ref{th1}}
The RHS of \eqref{sizeC}, modulo the $o(n)$ number of vertices that are spanned by non tree components in $\G_L$, is clearly an upper bound on the largest cycle in $C_2$. Any cycle must omit at least $\f(T)$ vertices from each $T\in\cT$. On the other hand, as we show, w.h.p. there is cycle $H$ that spans $V_1\cup \bigcup_{T\in \cT}V(\cQ(T))$ (see Notation 1). The length of $H$ is equal to the RHS of \eqref{sizeC}. Equivalently, we show that 
\beq{claim*}{
\text{w.h.p. there is a Hamilton cycle $H^*$ in $\G^*$ that contains all the edges of $M^*$.}
}

\subsection{Proof of \eqref{small}}
We are not able at this time to give a simple estimate of $\sum_{T\in\cT}\f(T)$ as a function of $c$. We will have to make do with \eqref{small}. On the other hand, $\sum_{T\in\cT}\f(T)$ can be approximated to within arbitrary accuracy, using the argument in Section \ref{seclimit}. 

We work in $G_{n,p}$. Observe that $T$ must have a vertex of degree three in order that $\f(T)>0$. The smallest such tree has seven vertices and consists of three paths of length two with a common 
endpoint. (If $T$ is a star of degree 3 for example, it can be covered by a path of length 2 that covers the central vertex and a path of length 0. Here we are using that every vertex in $V(T) \setminus V_2 \subset C_2$ must have degree at least 2, hence every vertex of $T$ of degree 1 belongs to $V_2$ and has neighbors in $C_2 \setminus V(T)$.) Therefore, in $G_{n,p}$,
\begin{align}
\E\brac{\sum_{T\in \cT}\f(T)}&\leq  \sum_{k\geq 7} k \cdot \binom{n}{k} k^{k-2}p^{k-1}(1-p)^{(n-k)\max\set{3,k/\ec}}\nn\\
&\leq  \sum_{k\geq 7}\bfrac{ne}{k}^k k^{k-1}\bfrac{c}{n}^{k-1}\exp\set{-c\max\set{3,k/\ec}}\nn\\
&=O(c^{6}e^{-3c})n,\label{few}
\end{align}
At the first line we used that every tree that contributes to
$\E\brac{\sum_{T\in \cT}\f(T)}$ must satisfy $v_0(T)>2$. In addition \eqref{vK} states that $v_0(T) \geq |T|/\ec$.  
We obtain \eqref{small} from \eqref{few}.
\subsection{Structure of $\G_1^*$}
Suppose now that $|V_1|=N$ and that $V_1$ contains $M$ edges. The construction of $\G_L$ does not involve the edges inside $V_1$, but we do know that that $\G_1^*$ has minimum degree at least $\ec$. The distribution of $\G_1^*$ will be that of $G_{V_1,M}$ subject to this degree condition, viz. the random graph $G_{V_1,M}^{\delta\geq\ec}$ which is sampled uniformly from the set $\cG_{V_1,M}^{\delta\geq\ec}$, the set of graphs with vertex set $V_1$, $M$ edges and minimum degree at least $\ec$. This is because, we can replace $\G_1^*$ by any graph in $G_{V_1,M}^{\delta\geq\ec}$ without changing $\G_L$. By the same token, we also know that each  $v\in V_2^*$ has at least $\ec$ random neighbors in $V_1$. We have that
\beq{MNsize}{
N\geq n(1-2e^{-c/2})\text{ and }M\in \frac{(1\pm\e_1)cN}{2},
}
where $\e_1=c^{-1/3}$. The bound on $N$ follows from \eqref{C2v} and \eqref{smallS} and the bound on $M$ follows from the fact that in $G_{n,p}$,
\[
\Pr\brac{\exists S:|S|=N,e(S)\notin(1\pm\e_1)\binom{N}{2}p}\leq 2\binom{n}{N}\exp\set{-\frac{\e_1^2N(N-1)p}{3}}=o(1).
\]
\subsection{Partitioning/Coloring $G=G_{n,p}$}
We will use the edge coloring argument of Fenner and Frieze \cite{FF} to verify \eqref{claim*}. In this section we describe how to color edges.

We color most of the edges of $G$ light blue, dark blue or green. We denote the resultant blue and green subgraphs by $\G_b, \G_g$ respectively (an edge is blue if it is either dark or light blue). We later show that the blue graph has  expansion properties while the green graph has suitable randomness.

Every vertex $v\in V_1$ independently chooses  $\min\set{\deg_{V_1}(v),100}$ neighbors in $V_1$ and we color the chosen edges light blue. Then we color every edge in $V_2^*:V_1$ light blue. Thereafter we independently color (re-color) every edge of $G$ dark blue with probability $1/2000$.  Finally we color green all the uncolored edges that are contained in $V_1$. (Some of the edges of $G$ will remain uncolored and play no significant role in the proof.)

The above coloring satisfies the following properties:
\begin{enumerate}[{\bf (C1)}]
\item Every vertex in  $V_1\cup V_2^*$ is joined to at least $\ec$ vertices in $V_1$ by a blue edge.
\item Every dark blue edge appears independently with probability $\frac{p}{2000}$.
\item Given the degree sequence $\bd_g$ of $\G_g$, every graph $H$ with vertex set $V_1$ and degree sequence $\bd_g$ is equally likely to be $\G_g$.  
\end{enumerate}

We can justify {\bf C3} as follows: Amending $G$ by replacing $\G_g$ by any other graph $G'$ with vertex set $V_1$ and the same degree sequence and executing our construction of $S_L$ will result in the same set $S_L$ and sets $V_1,V_2^*$. So, each possible $G'$ has the same set of extensions to $G_{n,p}$ and as such is equally likely.

Now given $\G_b, \G_g \subset G$ we color the edges in $\G^*$ as follows. Every edge in $\G^*$ that exists in $G$ inherits its color from the coloring in $G$. Every edge in $M^* \subseteq E(\G^*)$ is colored blue. We let $\G_b^*,\G_g^*$ be the blue and the green subgraphs of $\G^*$. Observe that $\G_g^*= \G_g$, hence $\G_g^*$ satisfies property $(C3)$ as well.
\subsection{Expansion of $\G_b^*$}\label{expand}
We wish to estimate the probability that small sets have relatively few neighbors in the graph $\G_b^*$. For  $S\subseteq V^*= V_1\cup V_2^*$ we let 
\begin{align*}
N_b (S)&=\set{w\in V_1 \setminus S :\exists v\in S \text{ with }\set{v,w}\in E(\G_b^*)}
\\&=\set{w\in V_1 \setminus S :\exists v\in S \text{ with }\set{v,w}\in E(\G_b)}
\end{align*}
We have slightly abused notation here since $N_b(S)$ is implicitly defined in both $G$ and $\G^*$.

It is shown in \cite{BCFF} and also in \cite{FPitt} that if $S$ is the set of endpoints created by P\'osa rotations (see Section \ref{Posa}) that $S\cup N(S)$ is connected and contains at least two distinct cycles hence, at least $|S|+ |N(S)|+1$ edges. Hence the condition (iii) in the following lemma.

\begin{lemma}\label{explem1}
W.h.p. there does not exist $S \subset V^*$ of size $|S| \leq n/4$ such that (i) $|N_b(S)| \leq 2|S|$, (ii) $S\cup N_b(S)$ is connected in $\G_b\subseteq G_{n,p}$ and (iii) $S\cup N_b(S)$ spans at least $|S|+|N_b(S)|+1$ edges in $\Gamma_b \subseteq G_{n,p}$.
\end{lemma}
\begin{proof}
Assume that the above fails for some set $S$. \\
{\bf Case 1: $|S|\leq n_1=n/(100c^3)$.}\\
Let $t=|N_b(S)|$. We will suppose first that $S$ contains at least $s/10$ vertices of degree at least 100. In this case $S\cup N_b(S)$ has cardinality at most $s+t\leq 3s$ and contains at least $5s>3(s+t)/2$ edges, contradicting Lemma \ref{lem1}.

On the other hand, if there are at least $9s/10$ vertices in $S$ of degree at most 99 then there are at least $3(s+t)/10$ vertices of degree at most 99 in a connected subgraph of size $s_0\leq s+t\leq 3n_1$. In addition that subgraph spans at least $s+t+1$. But the probability of this occuring in $G_{n,p}$ is at most
\multstar{
\sum_{k=1}^{3n_1}\binom{n}{k}k^{k-2}\binom{\binom{k}{2}}{2} p^{k+1}\binom{k}{3k/10} \brac{\sum_{\ell=1}^{99}\binom{n-k}{99}p^{\ell}(1-p)^{n-k-\ell}}^{3k/10}\\
\leq \sum_{k=1}^{3n_1}\bfrac{ne}{k}^kk^{k+2}\bfrac{c}{n}^{k+1}2^ke^{-3kc/20} 
\leq \sum_{k=1}^{3n_1} \frac{ck^2}{n} \cdot \brac{2ce^{1-3c/20}}^k=o(1).
}
This completes the proof for Case 1.

{\bf Case 2: $n_1<|S|\leq n/4$.}\\
The particular values for the sets $V_1,V_2^*$ condition $G_{n,p}$. To get round this, we describe a larger event $\cE_S$ in $G=G_{n,p}$ that (a) occurs as a consequence of there being a set $S$ with small expansion and (b) only occurs with probability $o(1)$. This event involves an arbitrary choice for $V_1,V_2^*$ etc.

Let $T=N_b(S)$ and $W =N_G(S) \setminus N_b(S)$, that is $T$ and $W$ are the neighborhood of $S$ inside and outside of $V_1$ respectively.  Then the following event $\cE_S$ must hold. There exist $S,T,W$ such that, where $s=|S|,t=|T|$ and $w=|W|$,
\begin{enumerate}[(i)]
\item $t \leq 2s$.
\item $w \leq n_0=ne^{-c/2}$, where $n_0$ is from \eqref{smallS}.
\item No vertex in $S$ is connected to a vertex in $V \setminus (S \cup T \cup W)$ by a dark blue edge.
\item $S \cup N_b(S)$ spans at least $s+t$ edges (at least s+t+1 in fact).
\end{enumerate}
Thus,
\begin{align*}
&\Pr(\cE_S\mid s,t,w)\nonumber\\
&\leq \binom{n}{s} \binom{n}{t}\binom{n}{w}  \binom{\binom{s+t}{2}}{s+t} s^w p^{s+t+w} \bigg(1-\frac{p}{2000}\bigg)^{s(n-s-t-w)} \nonumber
\\& \leq \bfrac{en}{s}^s \bfrac{en}{t}^t \bfrac{en}{w}^w  \bfrac{e(s+t)}{2}^{s+t} s^w  \bfrac{c}{n}^{s+t+w}     \exp\set{-\frac{p}{2000}\brac{\frac{sn}{5}}} \nonumber
\\& \leq (ec)^{2(s+t)}\bfrac{s+t}{2s}^s \bfrac{ s+t}{2t}^t \bfrac{ecs}{w}^w   \exp\set{-\frac{cs}{10^5}} \nonumber 
\\& \leq (ec)^{6s} \exp \set{ s \cdot \frac{t-s}{2s}} \exp \set{ t \cdot \frac{s-t}{2t}} 
 \bfrac{ecs}{n_0}^{n_0}    \exp\set{-\frac{cs}{10^5}} \nonumber 
\\& \leq  (ec)^{6s}  (ce^{1-c/3})^{se^{-c/3}}   \exp\set{-\frac{cs}{10^5}}   
= \bigg((ec)^6 (ce^{1-c/3})^{e^{-c/3}}  e^{-c/10^5} \bigg)^s .
\end{align*}
At the 5th line we used $\frac{s+t}{2s}=1+\frac{t-s}{2s}\leq\exp\set{\frac{t-s}{2s}}$ and $w\leq n_0\leq 100c^3e^{-c/2}s\leq e^{-c/3}s$. Hence
\[
\Pr(\exists S:\cE_S)\leq n\sum_{s=n/(100c^3)}^{n/4} \sum_{t=0}^{2s} \bigg((ec)^6(ce^{1-c/3})^{e^{-c/3}} e^{-c/10^5} \bigg)^s =o(1).
\]
\end{proof}

\subsection{The Degrees of the Green Subgraph}
\begin{lemma}\label{lem:green}
W.h.p. at least $99n/100$ vertices in $V_1$ have green degree at least $c/50$. In addition every set $S\subset V_1$ of size at least $n/4$ has total green degree at least $cn/250$.
\end{lemma}

\begin{proof}
At most $100n$ edges are colored light blue and thereafter the Chernoff bounds imply that w.h.p. at most $(1+\epsilon)cn/4000$ edges are colored dark blue, for some arbitrarily small positive $\e$. The probability that a vertex has degree less than $c/4$ in $G_{n,p}$  is bounded by $\frac{2e^{-c} \lambda^{c/4}}{c/4!}< 1/1000$. Azuma's inequality or the Chebyshev inequality can be employed to show that w.h.p. there are at most $n/1000$ vertices of degree less than $c/4$. Therefore every set of $n/100$ vertices is incident with at least $[(n/100-n/1000)c/4]/2$ edges. And hence with at least $ [(n/100-n/1000)c/4]/2 - (1+\epsilon)cn/4000-100 n\geq c/50 \cdot n/100$ green edges. Thus in every set of vertices of size at least $n/100$ there exists a vertex that is incident to $c/50$ green edges, proving the first part of our Lemma. 

It follows that w.h.p. every set of size $n/4$ has total green degree at least 
\[
\brac{\frac{n}{4}-\frac{n}{100}}\times \frac{c}{50}>\frac{cn}{250}.
\]
\end{proof}

\subsection{P\'osa Rotations}\label{Posa}
We say that a path/cycle $P$ in $\G^*$ is {\em compatible} if for every $\set{v,w}\in M^*$ either $P$ contains the edge $\set{v,w}$ or $V(P)\cap \set{v,w} = \emptyset$. Our aim therefore is to show that w.h.p. $\G^*$ contains a compatible hamilton cycle.
Suppose that $\G^*$ and hence $\G_b^*$ is not Hamiltonian and that $P=(v_1,v_2,\ldots,v_s)$ is a longest compatible path in both $\G^*$ and $\G_b^*$. If $\set{v_s,v_i} \in E(\G^*)$ and $v_i\in V_1$ then the path $P'=(v_1,v_2,\ldots,v_i,v_s,v_{s-1},\ldots,v_{i+1})$ is said to be obtained from $P$ by an {\em acceptable} rotation with $v_1$ as the fixed endpoint. We also call $v_i$ the {\em pivot vertex}  and the edges $\{v_s,v_i\}, \{v_i,v_{i+1}\}$  the {\em pivot edges}.  Observe that since $P$ is compatible and $\set{v_i,v_{i+1}} \notin M^*$  (since $v_i \in V_1$) then $P'$ is also  compatible. Let $END_b^*(P,v_1)$ be the set of vertices that are endpoints of paths that are obtainable from $P$ by a sequence of acceptable rotations with $v_1$ as the fixed endpoint. Then, for $v\in END_b^*(P,v_1)$ we let $END_b^*(P_v,v)$ be defined similarly. Here $P_v$ is a path with endpoints $v_1,v$ obtainable from $P$ by acceptable rotations.

Arguing as in the proof of P\'osa's lemma we see that $|N_b(END^*_b(P,v_1))|\leq 2|END_b^*(P,v_1)|$. 
Indeed, assume otherwise. Then there exist vertices $v_i,u$ such that $u \in END_b^*(P,v_1)$, $v_i \in N_b(u)\subseteq V_1$, $v_{i-1},v_{i+1}  \notin END_b^*(P,v_1)$ and the edge $\{u,v_i\}$ can be used by an acceptable rotation with $v_1$ as the fixed endpoint that ``rotates out" $u$. Any such rotation will create a path with either $v_{i-1}$ or $v_{i+1}$ as a new endpoint, say $v_{i-1}$. Now $v_i \in V_1$ and so the rotation will be acceptable and hence $v_{i-1} \in END_b^*(P,v_1)$ resulting in a contradiction. 

\begin{lemma}\label{explem}
W.h.p. for every path $P$ of maximal length in $\G_b^*$ and an endpoint  $v$ of $P$ we have that $|END_b^*(P_v,v)|\geq  n/4$.
\end{lemma}

Observe that the underlying graph in  Lemma \ref{explem1} is $\G_b$ and so we can not apply it directly to obtain Lemma \ref{explem}. In addition $\G_b^*$ is not a subgraph of $G_{n,p}$, since the edges in $M^*$ that are added correspond to paths in $G_{n,p}$.
\begin{proof}
We will show that $S=END_b^*(P_v,v)$ satisfies (i), (ii) and (iii) of Lemma \ref{explem1}. For this let $R=R(P_v,v)$ be the set of pivot points and $E_R=E_R(P)$ be the set of pivot edges. It is shown in \cite{BCFF} and also in \cite{FPitt} that if $S$ is the set of endpoints created by P\'osa rotations (see Section \ref{Posa}) then $E_R$ spans a connected subgraph on $S\cup R$ that consists of at least $|S|+|R\setminus S|+1$ edges.  

The key observation is that if $v$ is the pivot vertex of an acceptable rotation then, by definition, we have that $v\in V_1$. Consequently $R\subseteq V_1$ (i.e $R \subseteq N_b(S)$) and every edge in $E_R$ belongs to $E(\G_b) \subseteq E(G_{n,p})$. This would not have necessarily been true if $R \cap  V_2^* \neq \emptyset$. Finally, $(N_b(S)\setminus R):S$ spans at least $|N_b(S)\setminus R|$ edges in $\G_b$. Hence $N_b(S) \cup S$ is connected in $\G_b$ and spans at least $(|S|+|R\setminus S|+1)+|N_b(S)\setminus R| = |S|+|N_b(S)|+1$ edges. This verifies conditions (ii) and (iii) of Lemma \ref{explem1}. Condition (i) is satisfied by the discussion preceeding Lemma. \ref{explem}.
\end{proof}

From Lemma \ref{explem} we see that w.h.p. $|END_b^*(P_v,v)|\geq  n/4$ for all $v\in END_b^*(P,v_1)$. We let 
\[
END_b^*(P)=END_b^*(P,v_1)\cup \bigcup_{v\in END^*(P,v_1)}END_b^*(P_v,v).
\]

\subsection{Coloring argument}
We use a modification of a double counting argument that was first used in \cite{FF}. The specific version is from \cite{FF1}. Given a two edge-colored $\G^*$, we choose for each $v\in V_1$, an incident edge $\xi_v=\set{v,\eta_v}$ where $\eta_v\in V_1 \cup  V_2^*$. We re-color $\xi_v$ blue if it is not already colored blue. There are at most $\P=\prod_{v\in V_1}d(v)$ choices for $\bsx=(\xi_v,v\in V_1)$. 

For a graph $\G$, $\G=\G^*$ or $\G_b^*$, we let $\ell(\G)$ denote the length of the longest compatible path in $\G$. We indicate that $\G$ has a compatible Hamilton cycle by $\ell(\G)=|V(\G)|$.

We now let $a(\bsx,\G^*_g)=1$ if the following hold:
\begin{enumerate}[H1]
\item $\G_b^*$ is not Hamiltonian.
\item $\ell(\G_b^*)=\ell(\G^*)$.
\item $|N_b(S)|\geq 2|S|$ for all $S\subseteq V(\G^*), |S|\leq n/4$.
\end{enumerate}
We observe first that if $\G^*$ is not Hamiltonian and H2 holds then there exists $\bsx$ such that $a(\bsx,\G^*_g)=1$. Indeed, let $P=(v_1,v_2,\ldots,v_r)$ be a longest path in $\G^*$. Then we simply let $\xi_{v_i}$ be the edge $\set{v_i,v_{i+1}}$ for $1\leq i<r$.  It follows that if $\F$ denotes the number of  choices for $\G^*_g$ and $\p_{\bar{H}}$ is the probability that $\G^*$ is not Hamiltonian, then
\beq{A1}{
\p_{\bar{H}}\leq \frac{\displaystyle \sum_{\bsx,\G^*_g}a(\bsx,\G^*_g)}{\F}+o(1),
}
where the $o(1)$ term accounts for failure of the high probability events that we have identified so far.

On the other hand, we have as stated in (C3) above, that $\G^*_g$ is distributed as a random graph chosen uniformly from graphs with degree sequence $D^*_g$. Hence
\beq{A2}{
\sum_{\bsx,\G^*_g}a(\bsx,\G^*_g)\leq  \F\P\max_{\G_b^*}\p_b,
}
where $\p_b$ is defined as follows: let $P$ be some longest path in $\G^*_b$. Then $\p_g$ is the probability that a random realization of $\G_g^*$  does not include a pair $\set{x,y}$ where $y\in END_b^*(P,x)$.  We will argue below that
\begin{align}
\max_{\G_b}\p_b&\leq O(1)\times \prod_{v\in END_b^*(P)}\brac{1-\frac{ d_{\G^*_g}(v) {\displaystyle \sum_{w\in END_b^*(P_v,v)}d_{\G^*_g}(w)}}{ 2M}}^{\frac{1}{2}} \label{configx}\\
&\leq O(1)\times \exp\set{-\frac{{\displaystyle \sum_{v\in END_b^*(P)}  d_{\G^*_g}(v) {\displaystyle \sum_{w\in END_b^*(P_v,v)}d_{\G^*_g}(w)}}}{4M}}.\label{A3}
\end{align}
Lemma \ref{lem:green} implies that at least $n/4-n/100$ out of the at least $n/4$ vertices in  $END_b^*(P)$ have  $d_{\G^*_g}(v)\geq c/50$. Also, for such $v$ the set $END_b^*(P_v,v)\cup \{v\}$ is of size at least $n/4$ and so has total degree at least $cn/250$. Thus from \eqref{A3}, it follows that
$$ \max_{E_b}\p_g \leq O(1)\times \exp\set{-\frac{\frac{c}{50}\cdot (\frac{n}{4}-\frac{n}{100}) \cdot \frac{cn}{250} }{4M}} \leq e^{-cn/10^6}.$$

The Arithmetic-Geometric-mean inequality implies that 
\begin{align*}
\P\leq \prod_{v\in V_1} d(v)  \leq \bfrac{\sum_{v \in V} d(v)}{N}^N \leq (2c)^n
\end{align*}
It then follows that for sufficiently large $c$
\[
\p_{\overline{ H}}\leq (2c)^n \cdot e^{- cn/10^6 }+o(1)=o(1),
\]
and this completes the proof of \eqref{claim*}.

{\bf Proof of \eqref{configx}:} This is an exercise in the use of the configuration model of Bollob\'as \cite{B2}. Let $W=[2M_g]$ where $M_g$ is the number of green edges and let $W_1,W_2,\ldots,W_N$ be a partition of $W$ where $|W_v|=d_{\G^*_g}(v),v\in V_1$. The elements of $W$ will be referred to as {\em configuration points} or just as points.  A {\em configuration} $F$ is a partition of $W$ into $M_g$ pairs. Next define $\psi:W\to[N]$ by $x\in W_{\psi(x)}$. Given $F$, we let $\g(F)$ denote the (muti)graph with vertex set $V_1$ and an edge $\set{\psi(x),\psi(y)}$ for all $\set{x,y}\in F$. We say that $\g(F)$ is simple if it has no loops or multiple edges. Suppose that we choose $F$ at random. The properties of $F$ that we need are
\begin{enumerate}[{\bf P1}]
\item If $G_1,G_2\in \cG_{\bd_g}$ then $\Pr(\g(F)=G_1\mid \g(F)\text{ is simple})= \Pr(\g(F)=G_2\mid \g(F)\text{ is simple})$.
\item $\Pr(\g(F)\text{ is simple})=\Omega(1)$.
\end{enumerate}
These are well established properties of the configuration model, see for example Chapter 11 of \cite{FK}. Note that {\bf P2} uses the fact that w.h.p. $G_{V_1,M}^{\delta\geq 3}$ (and hence $\G_g^*$) has an exponential tail, as shown for example in \cite{hamd3}. Given all this, in the context of the configuration model, \eqref{configx} is a simple consequence of a random pairing of $W$. The $O(1)$ factor is $1/\Pr(\g(F)\text{ is simple})$ and bounds the effect of the conditioning. We take the square root to account for the possibility that $w\in END_b^*(P_v,v)$ and $v\in END_b^*(P_w,w)$.

\section{Proof of Theorem \ref{limit}}\label{seclimit}
For $v\in C_2$ we let $\phi(v)= \phi(T)/|\up_0(T)|$ if $v\in \up_0(T)$ for some $T \in \cT$ and $\phi(v)=0$ otherwise. (Recall that $\up_0(T)=V(T)\setminus V_2$.) Thus
$$\sum_{T \in \mathcal{T}}\phi(T)=\sum_{v \in C_2} \phi(v).$$

Hence \eqref{sizeC} can be rewritten as,
\begin{equation}\label{sizes2}
L_{c,n} \approx |C_2|- \sum_{v \in C_2} \phi(v).
\end{equation}

Let $k_1=k_1(\epsilon,c)$ be the smallest positive integer such that
\[
\sum_{k=k_1-1}^\infty (e^{\ec} 2^{\ec} ce^{-c/4})^k < \frac{\epsilon}{3}.
\]
Note that for large $c$, we have
\beq{sizek1}{
k_1\leq \frac{2}{c}\log\frac{1}\e.
}
 
For $v\in C_2$ let $G_v$ be the graph consisting of (i) the vertices of $G$ that are within distance $k_1$ from $v$  and (ii) a copy of $K_{\ec,\ec}$ where every vertex in the $k_1$ neighborhood of $v$ is adjacent to each vertex of the same one part of the bipartition. We consider the algorithm for the construction of $\Gamma_L$ on $G_v$ and let $C_{2,v},\Gamma_v, V_{1,v}, V_{2,v}, S_{L,v},\up_{0,v}(T)$ be the corresponding sets/quantities.

For a tree $T \in S_{L,v}$ let $f(T)$ be equal to $|T|$ minus the maximum number of vertices that can be covered by a set of vertex disjoint paths with endpoints in $V_{2,v}$ (we allow paths of length 0). For $v\in C_2$, if $v$ belongs to some tree $T \in S_{L,v}$ set $f(v)=f(T)/\up_{0,v}(T)$, otherwise set $f(v)=0$. 

For $v\in C_2$ let $t(v)=1$ if $v\in V_1$ or if $v\in S_L$ and in $\Gamma_L$, $v$ lies in a component with at most $k_1-2$ vertices that are not connected  to $V_1$ in $G$. Set $t(v)=0$ otherwise. Observe that if $t(v)=1$ then $\phi(v)=f(v)$. Otherwise $|\phi(v)-f(v)| \leq 1$. 

By repeating the arguments used to prove \eqref{11} and \eqref{vK} it follows that if 
$t(v)=0$ then $v$ lies on a component $C$ of size at most $\log n$. In addition at least $|V(C)|/\ec$ vertices in $V(C)$ are not adjacent to any vertex outside $V(C)$. Thus the expected number of vertices $v$ satisfying $t(v)=0$ is bounded by

\begin{align*}
&\sum_{k=k_1-1}^{\log^2 n} \sum_{j=k}^{\ec k}  \binom{n}{j} \binom{j}{k} j^{j-2}  p^{j-1} (1-p)^{k(n-j)}\\
& \leq  n \sum_{k=k_1-1}^{\log^2 n}  \ec k \bfrac{e}{{\ec} k}^{\ec k} 2^{{\ec} k}(\ec k)^{{\ec} k-2} c^{k-1}e^{-ck/4}\\ 
&\leq n  \sum_{k=k_1-1}^\infty (e^{\ec} 2^{\ec} ce^{-c/4})^k< \frac{\epsilon n}{3}.
\end{align*}

A vertex $v\in [n]$ is {\em good} if the $i$th level of its BFS neighborhood has size at most $3 c^i k_1/\epsilon$ for every $i\leq k_1$ and it is {\em bad} otherwise. Because the expected size of the $i^{th}$ neighborhood is $\approx c^i$ we have by the Markov inequality that $v$ is bad with probability at most $\approx \e/3k_1$ and so the expected number of bad vertices is bounded by $\e n/2$.  Thus
\begin{align*} 
\E\brac{\card{\sum_{v\in V}\phi(v)- \sum_{v\text{ is good }} f(v) }}
& \leq \E\brac{\card{\sum_{v\in V}\phi(v)- \sum_{v \in V} f(v)}}+ \E\brac{\card{\sum_{ v\text{ is bad }} f(v)}}\\
& \leq  \E\brac{\card{\sum_{v: t(v)=0}|\phi(v)-  f(v)}}+\E\brac{\sum_{v\text{ is bad }} 1}\\
& \leq \E\brac{\sum_{v: t(v)=0} 1} +\frac{\epsilon n}3\\
& \leq \frac{\epsilon n}2 +\frac{\epsilon n}3 < \epsilon n.
\end{align*}
Let $\mathcal{H}_\e$ be the set of BFS neighborhoods that are good i.e. whose $i$th levels are of size at most $3 c^i k_1/\epsilon$ for every $i\leq k_1$. Every element of $\mathcal{H}_\e$ corresponds to a pair $(H,o_H)$ where $H$ is a graph and $o$ is a distinguished vertex of $H$, that is considered to be the root.  Also for $v\in C_2$ let $G(N_{k_1}(v))$ be the  subgraph induced by the ${k_1}^{th}$ neighborhood of $v$. For $(H,o_H)\in \mathcal{H}_\e$ let $int(H)$ be the set of vertices incident to the first $k_1-1$ neighborhoods of $o_H$ and let $Aut(H,o_H)$ be the number of automorphisms of $H$ that fix $o_H$.  Note that each good vertex $v$ is associated with a pair $(H,o_H)\in\cH_\e$ from which we can compute $f(v)$, since $f(v)=f(o_H)$. Thus, if now $M=|E(C_2)|,N=|C_2|$,
\begin{align}
\E\brac{\sum_{v\text{ is good}}f(v)\bigg| M,N}&
=\sum_{v}\sum_{k\geq 1}\sum_{\substack{(H,o_H) \in \mathcal{H}_\e\\ (G(N_{k_1}(v)),v)=(H,o_H)\\|V(H)|=k}} \r_{H,o_H} f(o_H)\nonumber
\\&=o(n)+\sum_{v}\sum_{k\geq 1}\sum_{\substack{(H,o_H) \in \mathcal{H}_\e\\H\text{ is a tree}\\(G(N_{k_1}(v)),v)=(H,o_H)}} \r_{H,o_H} f(o_H),\label{goodvert}
\end{align}
where $\r_{H,\s_H}$ is the probability  $(G(N_{k_1}(v)),v)=(H,o_H)$ in $C_2$. We show in Section \ref{model} that
\beq{piH}{
\r_{H,o_H}\approx  \frac{1}{Aut(H,o_H)} \bfrac{N}{2M}^{k-1} \lambda^{2k-2} \frac{e^{k\l}}{f_2(\l)^k},
}
where $f_k$ is defined in \eqref{fk} below and $\l$ satisfies \eqref{2} below.

Finally observe that with the exception of the $o(1)$ term, all the terms in \eqref{goodvert} are independent of $n$. We let
\beq{hc}{
f_\e(c)= \sum_{k\geq 1}\sum_{\substack{(H,o_H) \in \mathcal{H}_\e\\H\text{ is a tree}}}  \frac{f(o_H)}{Aut(H,o_H)} \bfrac{N}{2M}^{k-1} \lambda^{2k-2} \frac{f_{2}(k\l)}{f_2(\l)^k}.
}
Then for a fixed $c$, we see that $f_\e(c)$ is monotone increasing as $\e\to 0$. This is simply because $\cH_\e$ grows. Furthermore, $f_\e(c)\leq 1$ and so the limit $f(c)=\lim_{\e\to0}f_\e(c)$ exists. This verifies part (a) of Theorem \ref{limit}. For part (b), we prove, (see \eqref{Azuma}),
\begin{lemma}
\[
\Pr(|L_{c,n}-\E(L_{c,n})| \geq \e n+n^{3/4}) = O(e^{-\Omega(n^{1/5})}).
\]
\end{lemma}
\begin{proof}
To prove this we show that if $\n(H)$ is the number of copies of $H$ in $C_2$ then $H\in \cH_\e$ implies that
\beq{nH}{
\Pr(|\n(H)-\E(\n(H))| \geq n^{3/5}) =O(e^{-\Omega(n^{1/5})}).
}
The inequality follows from a version of Azuma's inequality  (see \eqref{Azuma}), and the lemma follows from taking a union bound over 
\multstar{
\exp\set{O\bfrac{c^{k_1(\epsilon)}k_1(\epsilon)}{\epsilon}}=\exp\set{O\bfrac{c^{\frac{2\log \frac{1}{\e}}{c}}\frac{2\log \frac{1}{\e}}{c}}{\e}}\\
=\exp\set{O\bfrac{(1/\e)^{2\log c/c} \log \frac{1}{\e}}{c\e}}=\exp\set{O((1/\e)^{2+2\log c/c})}
}
graphs $H$.
 Note also that the $o(n)$ term in \eqref{goodvert} is bounded by the same $e^{O((1/\e)^{2+2\log c/c})}$ term times the number of cycles of length at most $2k_1$ in $G$. The probability that this exceeds $n^{1/2}$ is certainly at most the RHS of \eqref{nH}. We will give details of our use of the Azuma inequality in Section \ref{model}.
\end{proof}
Part (b) of Theorem \ref{limit} follows by letting $\e\to 0$ and from the Borel-Cantelli lemma.
\subsection{A Model of $C_2$}\label{model}
It is known that given $M,N$ that, up to relabeling vetices, $C_2$ is distributed as $G_{N,M}^{\delta\geq 2}$. The random graph  $G_{N,M}^{\delta\geq 2}$ is chosen uniformly from $\cG_{N,M}^{\delta\geq 2}$ which is the set of graphs with vertex set $[N]$, $M$ edges and minimum degree at least two.

\subsubsection{Random Sequence Model}\label{refined}
We must now take some time to explain the model we use for $G_{N,M}^{\delta\geq2}$. We use a variation on the pseudo-graph model of Bollob\'as and Frieze \cite{BollFr} and Chv\'atal \cite{Ch}. Given a sequence $\bx = (x_1,x_2,\ldots,x_{2M})\in [n]^{2M}$ of $2M$ integers between 1 and $N$ we can define a (multi)-graph
$G_{\bx}=G_\bx(N,M)$ with vertex set $[N]$ and edge set $\{(x_{2i-1},x_{2i}):1\leq i\leq M\}$. The degree $d_\bx(v)$ of $v\in [N]$ is given by 
$$d_\bx(v)=|\set{j\in [2M]:x_j=v}|.$$
If $\bx$ is chosen randomly from $[N]^{2M}$ then $G_{\bx}$ is close in distribution to $G_{N,M}$. Indeed,
conditional on being simple, $G_{\bx}$ is distributed as $G_{N,M}$. To see this, note that if $G_{\bx}$ is simple then it has vertex set $[N]$ and $M$ edges. Also, there are $M!2^M$ distinct equally likely values of $\bx$ which yield the same graph. 

Our situation is complicated by there being a lower bound of 2 on the minimum degree. So we let
$$[N]^{2M}_{\delta\geq 2}=\{\bx\in [N]^{2M}:d_\bx(j)\geq 2\text{ for }j\in[N]\}.$$
Let $G_\bx$ be the multi-graph $G_\bx$ for $\bx$ chosen uniformly from $[N]^{2M}_{\delta\geq2}$. It is clear then that conditional on being simple, $G_\bx$ has the same distribution as $G_{N,M}^{\d\geq 2}$. It is important therefore to estimate the probability that this graph is simple. For this and other reasons, we need to have an understanding of the degree sequence $d_\bx$ when $\bx$ is drawn uniformly from $[N]^{2M}_{\delta\geq2}$. Let 
\beq{fk}{
f_k(\l)=e^\l-\sum_{i=0}^{k-1}\frac{\l^i}{i!}
}
for $k\geq 0$.
\begin{lemma}
\label{lem3}
Let $\bx$ be chosen randomly from $[N]^{2M}_{\delta\geq2}$. Let $Z_j,j=1,2,\ldots,N$ be independent copies of a {\em truncated Poisson} random variable $\cP$, where
$$\Pr(\cP=t)=\frac{{\l}^t}{t!f_2({\l})},\hspace{1in}t\geq 2.$$
Here ${\l}$ satisfies
\begin{equation}\label{2}
\frac{{\l}f_{1}({\l})}{f_2({\l})}=\frac{2M}{N}.
\end{equation}
Then $\{d_\bx(j)\}_{j\in [N]}$ is distributed as $\{Z_j\}_{j\in [N]}$ conditional on $Z=\sum_{j\in [n]}Z_j=2M$.
\end{lemma}
\begin{proof}
This can be derived as in  Lemma 4 of \cite{AFP}.
\end{proof}
It follows from \eqref{MNsize} and \eqref{2} and the fact that $f_{1}({\l})/f_2({\l})\to 1$ as $c\to \infty$ that for large $c$,
\beq{uplam}{
\l=c\brac{1+O(ce^{-c})}.
}
We note that the variance $\s^2$ of $\cP$ is given by
\[
\s^2=\frac{\l(e^\l-1)^2-\l^3e^\l}{f_2^2(\l)}.
\]
Furthermore,
\begin{align}
\Pr\left(\sum_{j=1}^NZ_j=2M\right)&=\frac{1}{\s\sqrt{2\p N}}(1+O(N^{-1}\s^{-2}))\label{local1}\\
\noalign{and}
\Pr\left(\sum_{j=2}^NZ_j=2M-d\right)&=\frac{1}{\s\sqrt{2\p N}}\left(1+O((d^2+1)N^{-1}\s^{-2})\right). \label{local2}
\end{align}
This is an example of a local central limit theorem. See for example, (5) of \cite{AFP} or (3) of \cite{hamd3}. It follows by repeated application of \eqref{local1} and \eqref{local2} that if $k=O(1)$ and $d_1^2+\cdots+d_k^2=o(N)$ then
\beq{local3}{
\Pr\brac{Z_i=d_i,i=1,2,\ldots,k\mid\sum_{j=1}^NZ_j=2M}\approx \prod_{i=1}^k\frac{\l^{d_i}}{d_i!f_2(\l)}.
}
Let $\n_\bx(s)$ denote the number of vertices of degree $s$ in $G_\bx$. 
\begin{lemma}
\label{lem4x}
Suppose that $\log N=O((N {\l})^{1/2})$. Let $\bx$ be chosen randomly from $[N]^{2M}_{\delta\geq2}$. Then as in equation (7) of \cite{AFP}, we have that with probability $1-o(N^{-10})$,
\begin{align}
\left|\n_\bx(j)-\frac{N{\l}^j}{j!f_{2}({\l})}\right|& \leq \brac{1+\bfrac{N {\l}^j}{j!f_2({\l})}^{1/2}}\log^2 N,\ 2\leq j\leq \log N.\label{degconc}\\
\n_\bx(j)&=0,\quad j\geq \log N.\label{degconc1}
\end{align}
\end{lemma}
We can now show $G_\bx$, $\bx\in [n]^{2m}_{\delta\geq2}$ is a good model for $G_{n,m}^{\d\geq 2}$. For this we only need to show now that
\beq{simpx}{
\Pr(G_\bx\text{ is simple})=\Omega(1).
 }
Again, this follows as in \cite{AFP}. 

Given a tree $H$ with $k$ vertices of degrees $z_1,z_2,...,z_k$ and a fixed vertex $v$ we see that if $\r_H$ is the probability that $G(N_{k_1}(v))=H$  in $G_{\bx}$ then we have
\newpage
\begin{align}
\r_{H,o_H}&\approx \binom{N}{k-1} \frac{(k-1)!}{Aut(H,o_H)}\times\nonumber\\
& \hspace{.5in}\sum_{D=2k-2}^\infty\  \sum_{\substack{d_1\geq z_1,\ldots,d_k\geq z_k\\d_1+\cdots+d_k=D}} \prod_{i=1}^k\frac{\l^{d_i}}{d_i!f_2(\l)}\cdot\binom{M}{k-1}2^{k-1}(k-1)!\cdot\prod_{i=1}^k \frac{d_i!}{(d_i-z_i)!}\frac{1}{(2M)^{2k-2}}\label{sigma1}\\
&\approx\bfrac{N}{2M}^{k-1} \frac{\lambda^{2k-2}}{Aut(H,o_H)f_2(\l)^k}  \sum_{D=2k-2}^\infty\  \sum_{\substack{d_1\geq z_1,\ldots,d_k\geq z_k\\d_1+\cdots+d_k=D}} \prod_{i=1}^k\frac{\l^{d_i-z_i}}{(d_i-z_i)!} \nonumber\\
&=\bfrac{N}{2M}^{k-1}\frac{\lambda^{2k-2}}{Aut(H,o_H)f_2(\l)^k}   \sum_{D=2k-2}^\infty\frac{(k\l)^{D-2(k-1)}}{(D-2(k-1))!}\label{sigma2}\\
&\approx \frac{1}{Aut(H,o_H)} \bfrac{N}{2M}^{k-1} \lambda^{2k-2} \frac{e^{k\l}}{f_2(\l)^k}.\nonumber
\end{align}
{\bf Explanation for \eqref{sigma1}:} We use \eqref{local3} to obtain the probability that the degrees of $[k]$ are $d_1,\ldots,d_k$. This explains the product $\prod_{i=1}^k\frac{\l^{d_i}}{d_i!f_2(\l)}$. Implicit here is that $d_i=O(\log n)$, from \eqref{degconc1}. The contribution to the degree sum $D$ for $D\geq 2k\log n$ can therefore be shown to be negligible. We use the fact that $k$ is small to argue that w.h.p. $H$ is induced. We choose the vertices, other than $v$ in $\binom{N}{k-1}$ ways and then $\frac{(k-1)!}{Aut(H,o_H)}$ counts the number of copies of $H$ in $K_k$. We then choose the place in the sequence to put these edges in $\binom{M}{k-1}2^{k-1}(k-1)!$ ways. Finally note that the probability the $z_i$ occurrences of the $i$th vertex are as claimed is asymptotically equal to $\frac{d_i(d_i-1)\cdots (d_i-z_i+1)}{(2M)^{z_i}}$ and this explains the factor $\prod_{i=1}^k \frac{d_i!}{(d_i-z_i)!}\frac{1}{(2M)^{2k-2}}$.

{\bf Explanation for \eqref{sigma2}:} We use the identity 
\[
\sum_{\substack{d_1,\ldots,d_k\\d_1+\cdots+d_k=D}}\frac{D!}{d_1!\cdots d_k!}=k^D.
\]
It only remains to verify \eqref{nH}. It follows from the above that $\E(\n(H)\mid M,N)=\Omega(N)$. We first condition on a degree sequence \bx\ satisfying \eqref{degconc}. We then work in the associated configuration model. We can generate a configuration $F$ as a permutation of the multi-set $\set{d_i\times i:i\in [N]}$. Interchanging two elements in a permutation can only change $\n(H)$ by $O(1)$. We can therefore apply Azuma's inequality to show that 
\beq{Azuma}{
\Pr(|\n(H)-\E(\n(H))|\geq n^{3/5})=O(e^{-\Omega(n^{1/5})}).
} 
(Specifically we can use Lemma 11 of Frieze and Pittel \cite{FP1} or Section 3.2 of McDiarmid \cite{McD}.) This verifies \eqref{nH}.
\section{Summary and open problems}
We have derived an expression for the length of the longest path in $G_{n,p}$ that holds for large $c$ w.h.p. It would be interesting to have a more algebraic expression. Also, we could no doubt make this proof algorithmic, by using the arguments of Frieze and Haber \cite{FH}. It would be more interesting to do the analysis for small $c>1$. Applying the coupling of McDiarmid \cite{McD1} we see that the random digraph $D_{n,p},p=c/n$ contains a path at least as long as that given by the R.H.S. of \eqref{sizeC1}. It should be possible to improve this, just as Krivelevich, Lubetzky and Sudakov \cite{KLS1} did for the earlier result of \cite{Fpath}.

\end{document}